\documentclass[10pt]{amsart}

\usepackage{amsmath, amsfonts, amsthm, amssymb, graphicx, fullpage, enumerate, subcaption, float, hyperref}

\newtheorem{theorem}{Theorem}[section]        

\newtheorem{corollary}[theorem]{Corollary}

\newtheorem{proposition}[theorem]{Proposition}
    
\newtheorem*{main theorem}{Main Theorem}

\theoremstyle{remark}      
   
\theoremstyle{definition}  
\newtheorem{definition}[theorem]{Definition}   
      
\def\N{\mathbb{N}}

\def\Z{\mathbb{Z}}

\begin{document}
\title{Schur's Theorem in Integer Lattices} 
\author{Vishal Balaji, Andrew Lott, Alex Rice}
 
\begin{abstract} A standard proof of Schur's Theorem yields that any $r$-coloring of $\{1,2,\dots,R_r-1\}$ yields a monochromatic solution to $x+y=z$, where $R_r$ is the classical $r$-color Ramsey number, the minimum $N$ such that any $r$-coloring of a complete graph on $N$ vertices yields a monochromatic triangle. We explore generalizations and modifications of this result in higher dimensional integer lattices, showing in particular that if $k\geq d+1$, then any $r$-coloring of $\{1,2,\dots,R_r(k)^d-1\}^d$ yields a monochromatic solution to $x_1+\cdots+x_{k-1}=x_k$ with $\{x_1,\dots,x_d\}$ linearly independent, where $R_r(k)$ is the analogous Ramsey number in which triangles are replaced by complete graphs on $k$ vertices. We also obtain computational results and examples in the case $d=2$, $k=3$, and $r\in\{2,3,4\}$.

\end{abstract}

\address{Department of Mathematics, Millsaps College, Jackson, MS 39210}
\email{balajv1@millsaps.edu} 
\email{lottal@millsaps.edu}
\email{riceaj@millsaps.edu}

\maketitle 
\setlength{\parskip}{5pt}   

\section{Introduction}

 The following striking result of Schur \cite{Schur} dates back over a century. Prior to the statement, we quickly develop some notation and terminology: for $N\in \N$ we use $[N]$ to denote $\{1,2,\dots,N\}$, and for $r\in \N$ we refer to a partition of a set into $r$ pairwise disjoint subsets (which we think of as $r$ different \textit{colors}) as an \textit{$r$-coloring}. Finally, we refer to collections of elements that are all the same color as \textit{monochromatic}. 

\begin{theorem}[Schur's Theorem] \label{sthm} For every $r\in \N$, there exists $N\in \N$ such that every $r$-coloring of $[N]$ yields a monochromatic set of the form $\{x,y,x+y\}$, in other words a monochromatic solution to $x+y=z$. 
\end{theorem}

\noindent The existence of $N\in \N$ that satisfies the conclusion of Theorem \ref{sthm} implies the existence of a \textit{first} $N\in \N$ that satisfies said conclusion, in other words a ``breaking point" at which the desired pattern switches from non-guaranteed to guaranteed, which we refer to as the \textit{Schur number} $S(r)$. The only known Schur numbers are $S(1)=2$, $S(2)=5$, $S(3)=14$, $S(4)=45$, and $S(5)=161$. The last of these, announced in 2017, was determined by Heule \cite{sn5} using a SAT solver computation over two petabytes in size. 

Even though the former predates the latter, the slickest and now most standard proof of Theorem \ref{sthm} uses the following seminal result of Ramsey \cite{ramsey}. 

\begin{theorem}[Ramsey's Theorem]\label{ramthm} For $r,k\in \N$, there exists $N\in \N$ such that every $r$-coloring of the edges of  a complete graph on $N$ vertices yields a monochromatic complete subgraph on $k$ vertices. 
\end{theorem}

\noindent In a similar spirit to our definition of Schur numbers, we define the following class of \textit{Ramsey numbers}. 

\begin{definition}For $r,k\in \N$, let $R_r(k)$ denote the minimum $N\in \N$ such that every $r$-coloring of the edges of a complete graph on $N$ vertices yields a monochromatic complete subgraph on $k$ vertices.  \end{definition}

The following is a quantitative version of Schur's Theorem, generalized to any number of variables, which in particular is a special case of a result listed in item 6.2(i) of \cite{Radz}. We include the short, standard proof below for the sake of exposition. 
 
\begin{theorem}\label{skthm} Suppose $N,r,k\in \N$. If $N\geq R_r(k)-1$, then every $r$-coloring of $[N]$ yields a monochromatic solution to $x_1+\cdots+x_{k-1}=x_k.$
\end{theorem}

\begin{proof} Suppose $N,r,k\in \N$ with $N\geq R_r(k)-1$. Fix an arbitrary $r$-coloring of $[N]$. From this coloring, define an $r$-coloring of the edges of a complete graph on $N+1$ vertices by coloring the edge connecting vertex $i$ to vertex $j$ with the color assigned to the integer $|i-j|$. By definition of $R_r(k)$, there exist vertices $i_1<i_2<\cdots<i_k$ such that the edges connecting $i_{\ell+1}$ to $i_{\ell}$ for $1\leq \ell\leq k-1$, and the edge connecting $i_k$ to $i_1$, are all the same color.

\noindent Translating this back to the integer coloring, we have that $i_{\ell+1}-i_{\ell}$ for $1\leq \ell \leq k-1$, and $i_k-i_1$, are all the same color. Further, all of these numbers lie in $[N]$ and satisfy \begin{equation*} (i_2-i_1)+(i_3-i_2)+\cdots+(i_k-i_{k-1})=i_k-i_1. \qedhere \end{equation*} \end{proof}

Just as classical Schur numbers were defined from Theorem \ref{sthm}, one can define a generalized Schur number $S(r,k)$ as the minimum $N\in \N$ for which the conclusion of Theorem \ref{skthm} holds. For example, Boza, Mar\'{\i}n, Reveulta, and Sanz \cite{boza} determined the exact formula $S(3,k)=k^3-k^2-k-1$.

A natural question is whether a version of Theorem \ref{sthm}, or Theorem \ref{skthm}, holds not just in the integers, but in other additive settings such as $[N]^d$, a chunk of the $d$-dimensional integer lattice, for $d>1$. However, without any additional modifications, the extension to higher dimensions turns out to be disappointingly trivial. Specifically, if $N\geq S(r)$, then one can always find a monochromatic solution to $x+y=z$, which we call a \textit{Schur triple}, in an $r$-coloring of $[N]^d$ by just looking along the diagonal $\{(n,n,\dots,n): n\in [N]\}$ and reducing back to the one-dimensional case. Conversely, armed with an $r$-coloring of $[S(r)-1]$ free of Schur triples, one can color $[S(r)-1]^d$ based solely on the first coordinate, and the lack of Schur triples is preserved. In other words, the breaking point is exactly the same in all dimensions. 
 
The question becomes more interesting if we insist on a level of non-degeneracy from our solutions in higher dimensions. Our main result is the following. 

\begin{theorem}\label{main} Suppose $N,r,k,d \in \N$ with $k\geq d+1$. If $N\geq R_r(k)^d-1$, then every $r$-coloring of $[N]^d$ yields a monochromatic solution to $x_1+\cdots+x_{k-1}=x_{k} $ with $\{x_1,\dots,x_d\}$ linearly independent. 

\end{theorem} 

In addition, we approach the problem computationally in Section \ref{compsec}, establishing via SAT solver computations sharp results in the case $d=2$, $k=3$, $r\in \{2,3\}$, and a nontrivial lower bound in the case $d=2$, $k=3$, $r=4$. We conclude by displaying examples exhibiting the lower bound in each case. 

\section{Notation and preliminary observations}

We begin by defining a family of modified Schur numbers.

\begin{definition} For $r,k,d,j\in \N$ with $j\leq \min\{d,k-1\}$, let $S_{d,j}(r,k)$ denote the minimum $N\in \N$ such that every $r$-coloring of $[N]^d$ yields a monochromatic solution to $x_1+\cdots+x_{k-1}=x_k$ with $\{x_1,\dots,x_j\}$ linearly independent. In the absence of  any of the parameters $k$, $d$, or $j$ ($r$ is always included, and $j$ is never included without $d$), we assume their most classical value, specifically $k=3$, $d=j=1$. In particular, $S(r)$ and $S(r,k)$ are still as previously defined. We also let $S^*_d(r,k)=S_{d,d}(r,k)$. \end{definition}

\noindent Using our newly developed notation, Theorem \ref{main} is precisely the statement that \begin{equation*} S^*_d(r,k) \leq R_r(k)^d-1 \end{equation*} whenever $k\geq d+1$. Generalizing the idea mentioned in the introduction of looking for Schur triples along the diagonal, the following proposition says that, all other parameters fixed, raising the dimension can only lower the modified Schur number in question. This is the reason we consider the maximal linear independence case $j=d$ in our main result. 

\begin{proposition}\label{star} Suppose $r,k,d,j\in \N$ with $j\leq \min\{d,k-1\}$. If $S_{d,j}(r,k)$ exists, then $S_{d+1,j}(r,k)$ exists and satisfies $S_{d+1,j}(r,k) \leq S_{d,j}(r,k)$.
\end{proposition}

\begin{proof} Suppose $r,k,d,j\in \N$ with $j\leq \min\{d,k-1\}$, and assume $S_{d,j}(r,k)$ exists. Let $N=S_{d,j}(r,k)$, and for $x=(n_1,\dots,n_d)\in [N]^d$, let $\tilde{x}=(n_1,\dots,n_d,n_d)\in [N]^{d+1}$. Fix an arbitrary $r$-coloring of $[N]^{d+1}$, and define an $r$-coloring of $[N]^d$ by coloring $x\in [N]^d$ the same as $\tilde{x}\in [N]^{d+1}$. By definition of $S_{d,j}(r,k)$, there exists a monochromatic solution to $x_1+\cdots+x_{k-1}=x_k$ in $[N]^d$ with $\{x_1,\dots,x_j\}$ linearly independent. This lifts to a monochromatic solution $\tilde{x}_1+\cdots+\tilde{x}_{k-1}=\tilde{x}_k$ in $[N]^{d+1}$, and the linear independence is also preserved. Therefore, $S_{d+1,j}(r,k) \leq S_{d,j}(r,k)$.
\end{proof}

\noindent Applying Proposition \ref{star} inductively and then Theorem \ref{main} yields the following corollary. 

\begin{corollary} If $r,k,d,j\in \N$ and $j\leq \min\{d,k-1\}$, then $S_{d,j}(r,k)\leq  R_r(k)^j-1$.
\end{corollary}

\section{Proof of Theorem \ref{main}}   

We now establish our main result by appropriately adapting the proof of Theorem \ref{skthm}. 

\begin{proof} Suppose $r,k,d\in \N$ with $k\geq d+1$, let $M=R_r(k)$, and let $N=M^d-1$. For $1\leq i \leq M$, let $y_i=(i,i^2,\dots,i^d)$. Fix an arbitrary $r$-coloring of $[N]^d$. We mimic the standard proof of Theorem \ref{skthm} and use this coloring to define an $r$-coloring of the edges of a complete graph on $M$ vertices, coloring the edge connecting vertex $i$ to vertex $j$, for $i<j$, with the color assigned to $y_j-y_i\in [N]^d$. By definition of $R_r(k)$, there exist vertices $i_1<i_2<\cdots<i_k$ such that the edges connecting $i_{j+1}$ to $i_j$ for $1\leq j\leq k-1$, and the edge connecting $i_k$ to $i_1$, are all the same color. Translating this back to the integer lattice coloring, we have that $y_{i_{j+1}}-y_{i_j}$ for $1\leq j \leq k-1$, and $y_{i_k}-y_{i_1}$, are all the same color. Further, these vectors satisfy \begin{equation*} (y_{i_2}-y_{i_1})+(y_{i_3}-y_{i_2})+\cdots+(y_{i_k}-y_{i_{k-1}})=y_{i_k}-y_{i_1}. \end{equation*}
It remains to be shown that the vectors $\{y_{i_2}-y_{i_1},\dots,y_{i_{d+1}}-y_{i_{d}}\}$ are linearly independent. We establish this fact using the $d\times d$ matrix formed with the relevant vectors as its rows, in other words \begin{equation*} A= \begin{bmatrix}
    i_2-i_1		 & i_2^2-i_1^2	 	& \dots 	& i_2^d-i_1^d \\
    i_3-i_2   	 & i_3^2-i_2^2   	& \dots 	& i_3^d-i_2^d \\
    \vdots		 &\vdots			& \ddots		&\vdots\\
    i_{d+1}-i_d   & i_{d+1}^2-i_d^2  & \dots 	&i_{d+1}^d-i_d^d\\
\end{bmatrix}.\end{equation*}
First, we note that a sequence of determinant-preserving row operations transforms the $(d+1)\times(d+1)$ \textit{Vandermonde matrix} \begin{equation*}V= \begin{bmatrix}
    1		 &  i_1	 	&  i_1^2 & \dots 	& i_1^d \\
    1   	 & i_2   	& i_2^2 & \dots 	& i_2^d \\
    \vdots		 &\vdots & \vdots			& \ddots		&\vdots\\
    1   & 	i_{d+1}  & i_{d+1}^2 & \dots 	& i_{d+1}^d\\
\end{bmatrix} \end{equation*} into 
\begin{equation*} \begin{bmatrix}
    1 & i_1 & i_1^2 & \dots & i_1^d \\
	0 & i_2-i_1		 & i_2^2-i_1^2	 	& \dots 	& i_2^d-i_1^d \\
    0 & i_3-i_2   	 & i_3^2-i_2^2   	& \dots 	& i_3^d-i_2^d \\
    \vdots & \vdots		 &\vdots			& \ddots		&\vdots\\
    0 & i_{d+1}-i_d   & i_{d+1}^2-i_d^2  & \dots 	&i_{d+1}^d-i_d^d\\
\end{bmatrix},\end{equation*}
so in particular $\det(A)=\det(V)$. This determinant, called the \textit{Vandermonde determinant}, is known to be precisely $\prod_{1\leq j<\ell\leq d+1} (i_{\ell}-i_j)$, which is positive since our inputs are strictly increasing. For a more self-contained approach, suppose $\mathbf{a}=(a_0,\dots,a_d)$ is a row vector satisfying $V\mathbf{a}^T=\mathbf{0}$. This implies that $i_1,\dots,i_{d+1}$ are all roots of the polynomial $$p(x)=a_0+a_1x+\cdots+a_dx^d.$$ If $p(x)$ is nonzero, then it has at most $d$ roots, so if $i_1,\dots,i_{d+1}$ are all distinct, it must be the case that $a_0=a_1=\cdots=a_d=0$. Therefore, $V$ and $A$ are both nonsingular, which completes the proof. 
\end{proof}

\section{Examples and computations} \label{compsec}

In this section, we describe in a general way how to use a SAT solver to obtain exact values or lower bounds for $S_{d,j}(r,k)$, and operationalize the method for $d=j=2$, $k=3$ and $r\in \{2,3,4\}$. In that context, we refer to a  monochromatic solution to $x_1+\cdots+x_{k-1}=x_{k}$ in $\Z^d$ with $x_1,\dots,x_j$ linearly independent as a \textit{$j$-nondegenerate Schur $k$-tuple}, while we drop the $j$ descriptor if $j=k-1$. 

\noindent Consider an $r$-coloring of $[N]^d$ defined by $\Delta: [N]^d\rightarrow\{1,2,...,r\}.$ Following the lead of Boza, et al. \cite{boza} and Heule \cite{sn5}, we write a logical expression in conjunctive normal form (cnf) which is true if and only if $\Delta$ yields no monochromatic $j$-nondegenerate Schur $k$-tuple. For each $p\in[N]^d$ and for each $m\in\{1,2,...,r-1\}$, define a boolean variable $\phi_m(p)$ by 
  \[\phi_m(p)=
  \begin{cases}
                                   \text{True} &  \text{if } \Delta(p)=m \\
                                   \text{False} & \text{otherwise} \\
  \end{cases}.\]
(Note that $\Delta(p)=r$ when $\phi_1(p),\phi_2(p),...,\phi_{r-1}(p)$ are all false). In order to guarantee that $\Delta$ assigns exactly one color to each point, our cnf expression must include
\[\mathcal{D}=\bigwedge_{p\in[N]^d}\bigwedge_{i<j\leq{r-1}}\big(\neg\phi_i(p)\lor\neg\phi_j(p)\big).\]
Now, let $\mathcal{F}$ be the family of all $j$-nondegenerate Schur $k$-tuples in $[N]^d$. We observe that $\Delta$ yields no monochromatic $j$-nondegenerate Schur $k$-tuple if and only if for each $\{p_1,p_2,...,p_k\}\in{\mathcal{F}}$ and for each color $i\in\{1,2,...,r\}$, $\Delta$ assigns at least one of the $k$ points to a color that is not $i$. Thus, for each $\{p_1,p_2,...,p_k\}\in{\mathcal{F}}$ we include the expression 
\[\mathcal{C}_{\{p_1,p_2,...,,p_k\}}=\left(\bigwedge_{i\in[r-1]}\left(\bigvee_{j\in[k]}\neg\phi_i(p_j)\right)\right)\land \left(\bigvee_{i\in[r-1]}\bigvee_{j\in[k]}\phi_i(p_j)\right).\]
Now, set 
\[\mathcal{C}=\bigwedge_{\{p_1,p_2,...,p_k\}\in{\mathcal{F}}}\mathcal{C}_{\{p_1,p_2,...,p_k\}}.\]
Then, $\Delta$ induces a coloring with no $j$-nondegenerate Schur $k$-tuples if and only if 
\begin{equation} \label{iffexp} \mathcal{D}\land\mathcal{C}. \end{equation}
In the case $d=j=2$, $k=3$, and $r\in \{2,3,4\}$, we ran \eqref{iffexp} through the following SAT solvers, with some computations completed on the Maple super cluster at the University of Mississippi, access to which was generously provided by the Mississippi Center for Supercomputing Research:

\begin{itemize} \item \url{https://github.com/arminbiere/lingeling}

\item \url{https://github.com/arminbiere/cadical}

\item \url{https://github.com/arminbiere/kissat}

\item \url{https://github.com/msoos/cryptominisat}
\end{itemize}

\

\noindent Our computations yielded $S^{\ast}_2(2,3)=7$, $S^{\ast}_2(3,3)=18$, and $S^{\ast}_2(4,3)\geq 49$, as well as examples exhibiting the lower bound in each case displayed in the three figures below, with which we conclude our discussion.

\

\begin{figure} [H] 
		\includegraphics[scale=0.5]{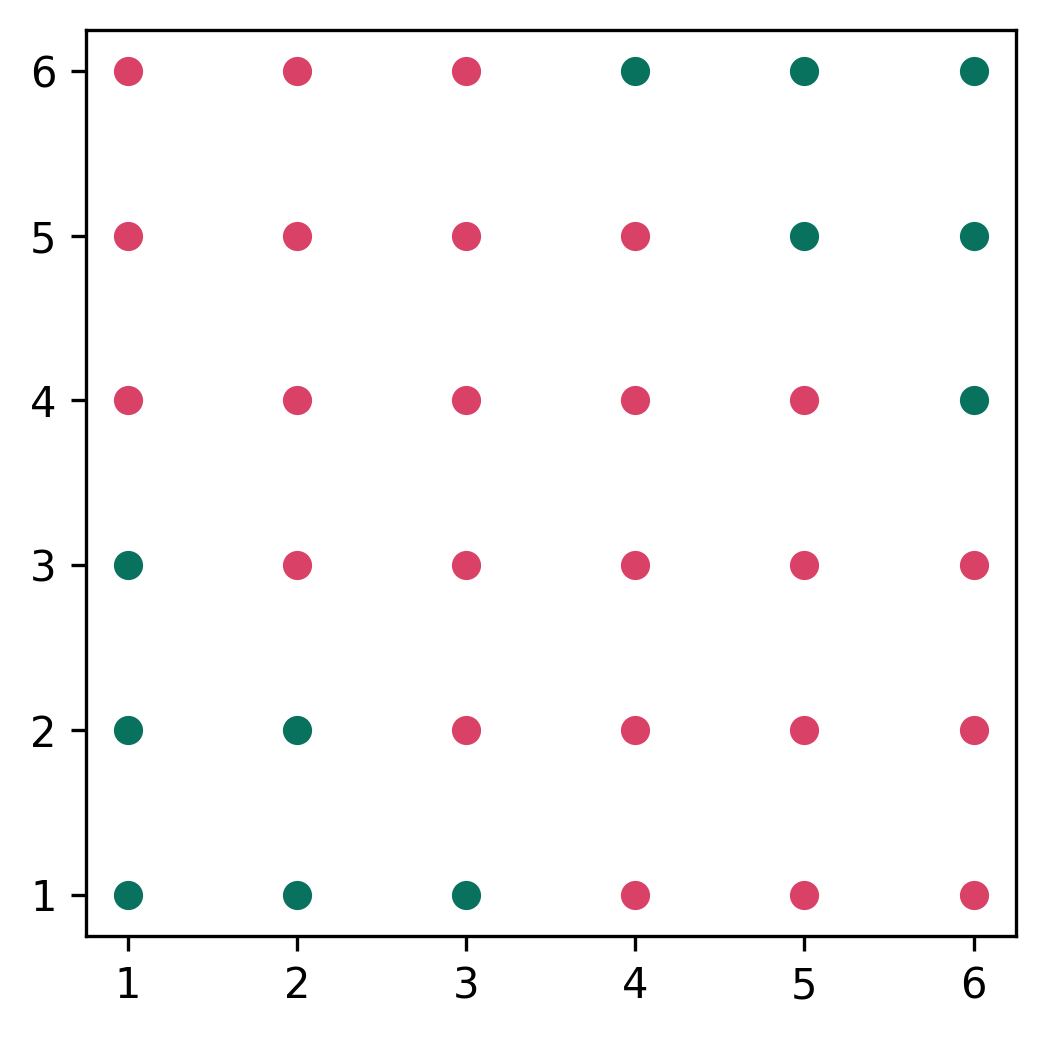} 
	 	\caption{A $2$-coloring of $[6]^2$ with no nondegenerate Schur triples.}
		 		
\end{figure}

\begin{figure} [H] 
		\includegraphics[scale=0.5]{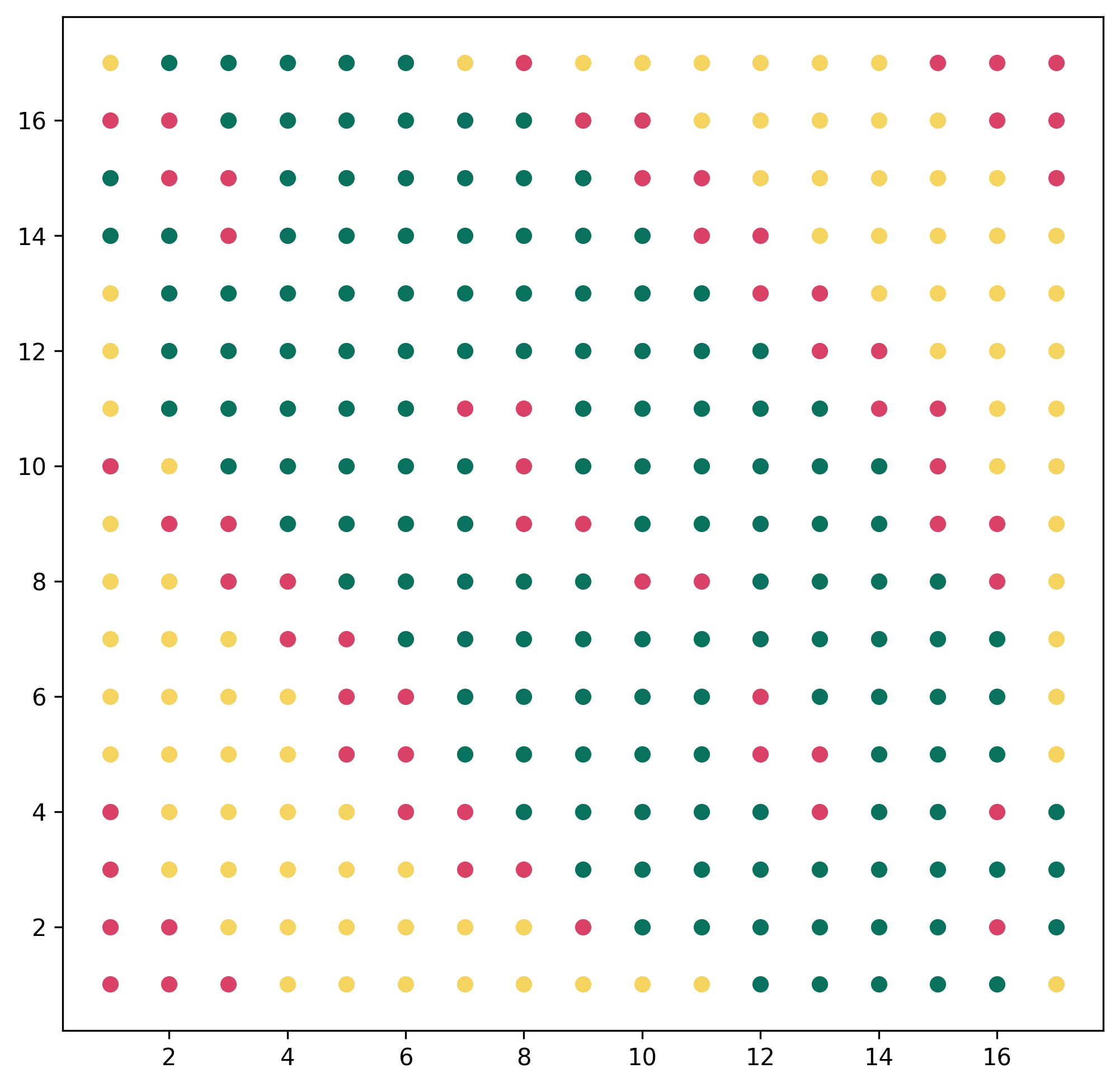} 
	 	\caption{A $3$-coloring of $[17]^2$ with no nondegenerate Schur triples.}
		 		
\end{figure}

\begin{figure} [H] 
		\includegraphics[scale=0.5]{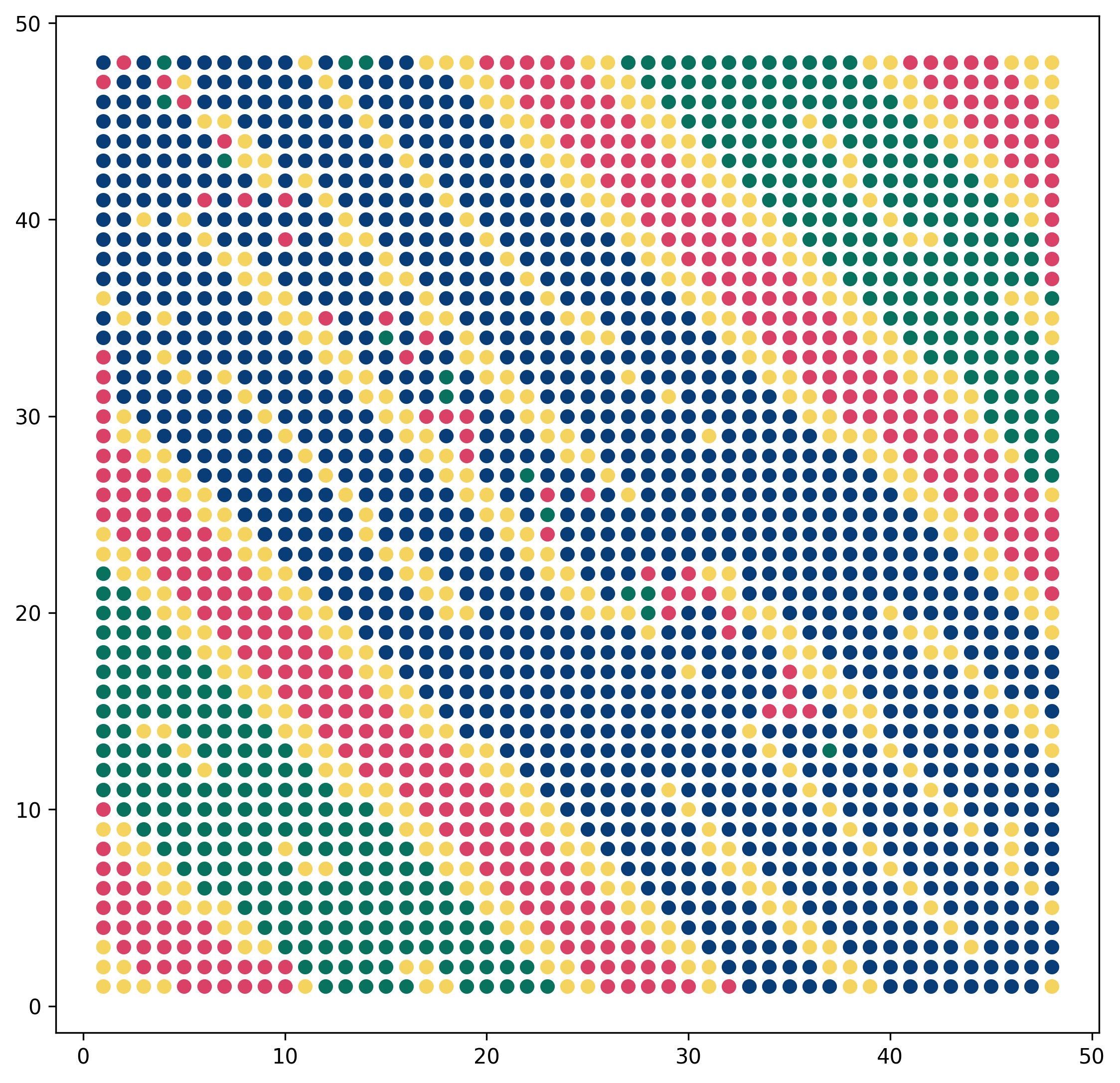} 
	 	\caption{A $4$-coloring of $[48]^2$ with no nondegenerate Schur triples.}
		 		
\end{figure} 
\newpage

\noindent \textbf{Acknowledgements:}  This research was initiated during the Summer 2021 Kinnaird Institute Research Experience at Millsaps College. All authors were supported during the summer by the Kinnaird Endowment, gifted to the Millsaps College Department of Mathematics. At the time of submission, Vishal Balaji and Andrew Lott were Millsaps College undergraduate students. The authors would like to thank the Mississippi Center for Supercomputing Research (MCSR), specifically MCSR director Benjamin Pharr, for allowing and assisting our computations on their Maple super cluster.

\end{document}